\documentclass[journal]{IEEEtran}
\usepackage{amsfonts}
\usepackage{amsmath}
\usepackage{amssymb}
\usepackage{xcolor}
\usepackage{graphicx}
\usepackage{cite}
\usepackage{array}
\usepackage{subfigure}
\usepackage{booktabs}
\usepackage{enumerate}
\usepackage{color}
\usepackage{algpseudocode}
\usepackage{epstopdf}
\usepackage{ntheorem}
\usepackage{pifont}
\usepackage{graphicx}
\usepackage{setspace}
\ifCLASSINFOpdf
\else
\fi
\newtheorem{theorem}{Theorem}

\newtheorem{assumption}{Assumption}

\newtheorem{lemma}{Lemma}

\newtheorem{problem}{Problem}
\newtheorem{proposition}{Proposition}

\newenvironment{proof}[1][Proof]{\noindent \textbf{#1.} }{\  \rule{0.5em}{0.5em}}
\ifCLASSINFOpdf \else \fi
\allowdisplaybreaks[4]

\begin{document}
\begin{spacing}{1.0}
\title{Global prescribed-time control of a class of uncertain nonholonomic systems by smooth
time-varying feedback}
\author{Kang-Kang Zhang, Bin Zhou, Chenchen Fan, James Lam, ~\IEEEmembership{Fellow,~IEEE}
\thanks{%
This work was supported by the National Science Found for
Distinguished Young Scholars (62125303), the Science Center Program of
National Natural Science Foundation of China (62188101), the
Fundamental Research Funds for the Central Universities
(HIT.BRET.2021008), and HKU CRCG (2302101740). \textit{(Corresponding authors: Bin Zhou)}}
\thanks{Kang-Kang Zhang is with the Department of Mechanical Engineering, University of Hong Kong,
Hong Kong, China, and the Department of Computer Science, KU Leuven, B-3001 Heverlee, Belgium; James Lam is with the Department of Mechanical Engineering, University of Hong Kong,
Hong Kong, China;
Bin Zhou is with the Center for Control Theory and Guidance Technology, Harbin Institute of Technology, Harbin, 150001, China; Chenchen Fan is with the Department of Rehabilitation Sciences, The Hong Kong Polytechnic University, Hong Kong, China  (email: kang-kang.zhang@kuleuven.be, binzhou@hit.edu.cn, chenchen.fan@polyu.edu.hk, james.lam@hku.hk).}
}
\maketitle

\begin{abstract}
This paper investigates the prescribed-time smooth control problem for a class of
uncertain nonholonomic systems.  With a novel smooth time-varying state transformation, the uncertain chained nonholonomic system is reformulated as an uncertain linear time-varying system.
By fully utilizing the properties of a class of parametric Lyapunov equations and constructing time-varying Lyapunov-like functions,
smooth time-varying high-gain state and output feedback controllers are designed. The states and controllers are proven to converge to zero at any prescribed time. The proposed smooth time-varying method combines the advantage of a
time-varying high-gain function, which enhances control performance, and a
smooth time-varying function that can drive the states to zero at the prescribed time. The
effectiveness of the proposed methods is verified by a numerical example.
\end{abstract}

\markboth{}{Shell
\MakeLowercase{\textit{et al.}}: Bare Demo of IEEEtran.cls for IEEE Journals}

\IEEEpeerreviewmaketitle

\section{Introduction}

Since nonholonomic systems have a wide range of applications, such as wheeled
mobile robots and surface vessels, the control problem of nonholonomic systems
has received widespread attention
\cite{Du15auto,Ge03auto,Hongtac05,Jiang00auto,Murry93tac}. However, nonholonomic systems
do not satisfy the necessary condition for the existence of a smooth
time-invariant feedback for a nonlinear system, called Brockett's necessary
conditions \cite{Brockett83}. Therefore researchers can only turn their
attention to discontinuous time-invariant control/hybrid control and smooth
time-varying control. For discontinuous time-invariant control, a general strategy for designing discontinuous control laws that can achieve asymptotic convergence for a class of nonholonomic control systems was proposed in \cite{Astolfi96scl}. For smooth
time-varying control, a control law which globally asymptotically stabilizes
nonholonomic system in chained/power form was presented in \cite{Teel95ijc}. In order to improve the convergence rate,
by using a time-varying state transformation, a smooth time-varying feedback
control law that can achieve exponential stability was proposed in
\cite{Tian00ieee,Tian02auto}. However, the above control laws can only ensure
that the states converge to zero asymptotically (exponentially), which
sometimes cannot meet the requirement in practical engineering.

\vspace{-0.1mm}
Compared with traditional asymptotic stability (or exponential stability),
finite-/fixed-time control has higher convergence speed, faster
convergence accuracy and stronger robustness to external disturbances
\cite{BBsiam00,Hongtac05,Sun19auto}. However, for both finite-time control and
fixed-time control, the actual convergence time (rather than the upper bound)
depends on the initial condition \cite{zhou21tac}. Recently, a so-called prescribed-time control method, whose convergence time does
not depend on the initial state, have received renewed attention
\cite{Huatac19,Han18tac,Song17auto,Zhang22scl,Zhou20autoa}. By using
time-varying high-gain feedback, the prescribed-time control problem of
normal-form nonlinear systems was solved in \cite{Song17auto}. Such time-varying high-gain methods have been utilized to address various
problems, such as prescribed-time stabilization \cite{Kri20auto,Li21tac},
prescribed-time observer \cite{esp22auto,Holl19tac}, adaptive prescribed-time
control \cite{Huatac19,Ye23scl}, prescribed-time output feedback
\cite{Holl19auto,Kri20ejc,Li23tac}, and prescribed-time tracking
control \cite{Ding23auto,Ye22tac}. With the aid of properties of a class of
parametric Lyapunov equations (PLEs) and the scalarization technique
\cite{zhou21tac}, the finite-time and prescribed-time control problems for
linear and a class of nonlinear systems were addressed in
\cite{zhang2024auto,Zhou20autoa,zhou21tac} based on time-varying
high-gain feedback. Very recently, a new approach called time-space
deformation methods has been proposed in \cite{Orlov22auto} to address
numerical singular problems that may arise in time-varying high-gain feedback.
To the best of the authors' knowledge, there are no smooth time-varying
feedback results for the prescribed-time control of nonholonomic systems.

In this paper, the smooth prescribed-time control problem for a class of
uncertain nonholonomic systems will be investigated. Firstly, with a novel smooth time-varying state
transformation, uncertain chained nonholonomic systems are reformulated as
uncertain linear time-varying systems. Secondly, by fully utilizing the
properties of a class of parametric Lyapunov equations and constructing
time-varying Lyapunov-like functions, smooth time-varying high-gain feedback
controllers are designed. It is proved that the states and controllers
converge to zero at any prescribed time. Both state feedback and
observer-based output feedback are considered.

The contributions of this paper are twofold. On the one hand, different from
the traditional discontinuous control methods
\cite{Astolfi96scl,Ge03auto,Hongtac05,Hou21ast,LQtac02,Han18tac,Zhang22scl,Zhang20tac,zhou23auto}%
, which may cause singularity or chattering problems, the proposed methods in
this paper are smooth. On the other hand, different from
\cite{Murry93tac,Teel95ijc,Tian00ieee,Tian02auto,zhou24gnc}, where only the state
feedback controllers were constructed to ensure that the state converges to zero
exponentially, the proposed controllers are not only based on state feedback
but also on output feedback, and can guarantee that the state converges to zero
at any prescribed time.

\textbf{Notation:} For a (symmetric) matrix $M$, let $M^{\mathrm{T}}$,
$\left \Vert M\right \Vert $, $\lambda_{\max}(M)$, and $\lambda_{\min}(M)$
denote its transpose, 2-norm, maximal eigenvalue, and minimal eigenvalue,
respectively. In addition, $\mathrm{diag} \{A_{1},\ldots, A_{n}\}$ denotes a
diagonal matrix whose $i$th diagonal element is $A_{i}$ .

\section{\label{sec1}Problem Introduction and Preliminaries}

\subsection{Problem Introduction}

Consider the following uncertain nonholonomic system%
\begin{equation}
\left \{
\begin{array}
[c]{rl}%
\dot{x}_{0} & =u_{0},\\
\dot{x}_{1} & =u_{0}x_{2}+\phi_{1}\left(  t,u,x\right)  ,\\
& \vdots \\
\dot{x}_{n-1} & =u_{0}x_{n}+\phi_{n-1}\left(  t,u,x\right)  ,\\
\dot{x}_{n} & =u+\phi_{n}\left(  t,u,x\right)  ,\\
y & =[x_{0},x_{1}]^{\mathrm{T}},
\end{array}
\right.  \label{sys}%
\end{equation}
where $[x_{0},x^{\mathrm{T}}]^{\mathrm{T}}=[x_{0},x_{1},\ldots,x_{n}%
]^{\mathrm{T}}\in \mathbf{R}^{n+1}$ is the state vector, $u_{0}\in \mathbf{R}$
and $u\in \mathbf{R}$ are two control inputs, $y\in \mathbf{R}^{2}$ is the
output vector, and $\phi_{i}(t,u,x), i=1,2,\ldots,n$ are some unknown continuous
functions satisfying the following assumption.

\begin{assumption}
\label{ass1}The nonlinear functions $\phi_{i}(t,u,x), i=1,2,\ldots,n,$ satisfy
\[
\left \vert \phi_{i}\left(  t,u,x\right)  \right \vert \leq c_{i1}\left \vert
x_{1}\right \vert +c_{i2}\left \vert x_{2}\right \vert +\cdots+c_{ii}\left \vert
x_{i}\right \vert,
\]
where $c_{ij},$ $j=1,2,\ldots,i$ are some known constants.
\end{assumption}

In this paper, regarding system (\ref{sys}), we are interested in the
following problem.

\begin{problem}
\label{prob2} Suppose that Assumption \ref{ass1} holds true. Let $T>0$ be a
prescribed time. Design a smooth time-varying state feedback controllers
$u_{0}(t)$ and $u(t)$ such that, for any initial condition, the states and
controls converge to zero at the prescribed time $T$.
\end{problem}

Since not all states are actually measurable, it is necessary to study
observer-based output feedback control strategies.

\begin{problem}
\label{prob3}Suppose that Assumption \ref{ass1} holds true. Let $T>0$ be a
prescribed time. Design a smooth time-varying output feedback
\[
\left \{
\begin{array}
[c]{rl}%
\dot{\xi}\left(  t\right)  & =F\xi \left(  t\right)  +Gu\left(  t\right)
+L\left(  t\right)  y\left(  t\right)  ,\\
u\left(  t\right)  & =H\left(  t\right)  \xi \left(  t\right)  ,
\end{array}
\right.
\]
such that, for any initial condition, the states and controls converge to zero
at the prescribed time $T$.
\end{problem}

\subsection{Preliminaries}

\subsubsection{Some properties of $\phi(t,u,x)$}

Consider the time-varying state transformation%
\begin{equation}
z\left(  t\right)  =L_{n}\left(  \frac{1}{T-t}\right)  x\left(  t\right)
,\quad t\in \lbrack0,T), \label{z_i_x_i}%
\end{equation}
where
\begin{equation}
L_{n}\left(  \gamma \right)  =\mathrm{diag}\left \{  \gamma^{n-1},\gamma
^{n-2},\ldots,1\right \}  . \label{eqdn}%
\end{equation}
For later use, denote, for $i=1,2,\dots,n$,%
\begin{align}
\varphi_{i}(t,u,x)  &  =\left(  T-t\right)  ^{i-n}\phi_{i}(t,u,x),\quad
\label{phis}\\
\psi_{i}(t,u,x)  &  =\varphi_{i}(t,u,x)+\frac{n-i}{T-t}z_{i}. \label{psi}%
\end{align}
Therefore, by using Assumption \ref{ass1} and (\ref{z_i_x_i}), it can be
obtained that, for $i=1,2,\dots,n$,%
\begin{align*}
\left \vert \varphi_{i}(t,u,x)\right \vert =  &  \left(  T-t\right)
^{i-n}\left \vert \phi_{i}(t,u,x)\right \vert \\
\leq &  \left(  T-t\right)  ^{i-n}\sum_{j=1}^{i}c_{ij}\left \vert
x_{j}\right \vert \\
=  &  \left(  T-t\right)  ^{i-n}\sum_{j=1}^{i}c_{ij}\left \vert \left(
T-t\right)  ^{n-j}z_{j}\right \vert \\
=  &  \sum_{j=1}^{i}c_{ij}\left \vert \left(  T-t\right)  ^{i-j}z_{j}%
\right \vert ,
\end{align*}
and further%
\begin{align}
|\psi_{i}(t,u,z)|\leq &  \sum_{j=1}^{i}c_{ij}|(T-t)^{i-j}z_{j}|+\frac
{n-i}{T-t}|z_{i}|\nonumber \\
=  &  \frac{1}{T-t}\sum_{j=1}^{i}c_{ij}|(T-t)^{i-j+1}z_{j}|+\frac{n-i}%
{T-t}|z_{i}|\nonumber \\
\leq &  \frac{1}{T-t}\sum_{j=1}^{i}g_{ij}\left \vert z_{j}\right \vert ,
\label{phi}%
\end{align}
where $g_{ij}=c_{ij}T^{i+1-j}$, $j=1,2,\ldots,i-1$, $i=2,3,\dots,n$,
and$\quad g_{ii}=c_{ii}T+n-i$, $i=1,2,\dots,n$.

With above preparations, we can state the following result by referring Lemma
1 in \cite{zhou21tac}.

\begin{lemma}
\label{lem1}Suppose that Assumption \ref{ass1} holds true. Then, for any
$z\in \mathbf{R}^{n}$ and any $t\in \lbrack0,T)$, there holds
\[
\mathit{\Phi}^{2}\left(  t,u,z\right)  \triangleq \left(  L_{n}\psi \right)
^{\mathrm{T}}\left(  L_{n}\psi \right)  \leq d^{2}\gamma^{2}\left(
L_{n}z\right)  ^{\mathrm{T}}\left(  L_{n}z\right)  ,
\]
where $\gamma=\frac{1}{T-t}$, $\gamma_{0}=\frac{1}{T}$, $L_{n}=L_{n}(\gamma)$,
and
\begin{equation}
d^{2}=\max \left \{
{\displaystyle \sum \limits_{i=1}^{n}}
\frac{g_{i1}^{2}i}{\gamma_{0}^{2\left(  i-1\right)  }},\cdots,%
{\displaystyle \sum \limits_{i=n}^{n}}
\frac{g_{in}^{2}i}{\gamma_{0}^{2\left(  i-n\right)  }}\right \}  . \label{d}%
\end{equation}

\end{lemma}

\begin{proof}
It follows from (\ref{phi}) that%
\begin{align*}
\psi_{i}^{2}\left(  t,u,z\right)   &  \leq \frac{1}{\left(  T-t\right)  ^{2}%
}\left(
{\displaystyle \sum \limits_{j=1}^{i}}
g_{ij}\left \vert z_{j}\right \vert \right)  ^{2}\\
&  \leq \frac{i}{\left(  T-t\right)  ^{2}}%
{\displaystyle \sum \limits_{j=1}^{i}}
g_{ij}^{2}\left \vert z_{j}\right \vert ^{2}.
\end{align*}
Then in view of (\ref{eqdn}) and (\ref{d}), we have%
\begin{align*}
&  \mathit{\Phi}^{2}\left(  t,u,z\right)  =\left(  L_{n}\psi \right)
^{\mathrm{T}}\left(  L_{n}\psi \right) \\
&  =%
{\displaystyle \sum \limits_{i=1}^{n}}
\gamma^{2\left(  n-i\right)  }\psi_{i}^{2}\left(  t,u,z\right) \\
&  \leq \frac{1}{\left(  T-t\right)  ^{2}}%
{\displaystyle \sum \limits_{i=1}^{n}}
\gamma^{2\left(  n-i\right)  }i%
{\displaystyle \sum \limits_{j=1}^{i}}
g_{ij}^{2}\left \vert z_{j}\right \vert ^{2}\\
&  \leq d^{2}\gamma^{2}\left(  \left \vert \gamma^{n-1}z_{1}\right \vert
^{2}+\left \vert \gamma^{n-2}z_{2}\right \vert ^{2}+\cdots+\left \vert \gamma
^{0}z_{n}\right \vert ^{2}\right)  .
\end{align*}
The proof is finished.
\end{proof}

\subsubsection{Some Properties of the PLE}

We now introduce some properties of the parametric Lyapunov equations (PLEs)
that will be used in this paper. Denote%
\begin{equation}
A=\left[
\begin{array}
[c]{c|cccc}%
0 & 1 &  &  & \\
\vdots &  & \ddots &  & \\
0 &  &  & 1 & \\ \hline
0 & 0 & \cdots & 0 &
\end{array}
\right]  ,\quad b=\left[
\begin{array}
[c]{c}%
0\\
\vdots \\
0\\\hline
1
\end{array}
\right]  ,\quad c=\left[
\begin{array}
[c]{c}%
1\\\hline
0\\
\vdots \\
0
\end{array}
\right]  ^{\mathrm{T}}. \label{eqab}%
\end{equation}
Consider the following PLE \cite{zhou21tac}
\begin{equation}
A^{\mathrm{T}}P+PA-Pbb^{\mathrm{T}}P=-\gamma P, \label{ple1}%
\end{equation}
and its dual form%
\begin{equation}
AQ+QA^{\mathrm{T}}-Qc^{\mathrm{T}}cQ=-\gamma Q, \label{ple2}%
\end{equation}
where $\gamma>0$ is a parameter to be designed.

\begin{lemma}
\label{lm1}\cite{zhou14book,zhou21tac} Let $(A,b)\in(\mathbf{R}^{n\times n}%
,\mathbf{R}^{n\times1})$ be given by (\ref{eqab}) and $\gamma>0$.

\begin{enumerate}
\item The PLE (\ref{ple1}) has a unique positive definite solution
\begin{equation}
P(\gamma)=\gamma L_{n}P_{n}L_{n}, \label{P}%
\end{equation}
where $P_{n}=P(1)$ and $L_{n}=L_{n}(\gamma)$.

\item The solution $P(\gamma)$ satisfies
\begin{equation}
\frac{\mathrm{d}P(\gamma)}{\mathrm{d}\gamma}>0,\quad b^{\mathrm{T}}%
P(\gamma)b=n\gamma. \label{bb}%
\end{equation}

\item There holds%
\begin{equation}
\frac{P(\gamma)}{n\gamma}\leq \frac{\mathrm{d}P(\gamma)}{\mathrm{d}\gamma}%
\leq \frac{\delta_{\mathrm{c}}P(\gamma)}{n\gamma},\quad \forall \gamma>0,
\label{dP_bounded}%
\end{equation}
where $\delta_{\mathrm{c}}=n(1+\lambda_{\max}(E_{n}+P_{n}E_{n}P_{n}^{-1}))$
with $E_{n}=\mathrm{diag}\{n-1,n-2,\ldots,1,0\}$.

\item There holds%
\begin{align}
A^{\mathrm{T}}PA\leq3n^{2}\gamma^{2}P. \label{hh}%
\end{align}

\end{enumerate}
\end{lemma}

\begin{lemma}
\label{lm3}\cite{zhou21tac} Let $( A,c) \in( \mathbf{R}^{n\times n}%
,\mathbf{R}^{1\times n}) $ be given by (\ref{eqab}) and $\gamma>0$.

\begin{enumerate}
\item The unique positive definite solution $Q(\gamma)$ is given by%
\begin{equation}
Q(\gamma)=\gamma^{2n-1}L_{n}^{-1}Q_{n}L_{n}^{-1}, \label{Q}%
\end{equation}
where $Q_{n}=Q(1)$ and $L_{n}=L_{n}(\gamma)$.

\item There holds%
\begin{equation}
\frac{Q( \gamma) }{n\gamma}\leq \frac{\mathrm{d}Q( \gamma) }{\mathrm{d}\gamma
}\leq \frac{\delta_{\mathrm{o}}Q( \gamma) }{n\gamma},\quad \forall \gamma>0,
\label{eqdq}%
\end{equation}
with $\delta_{\mathrm{o}}$ is a constant.

\item The matrices $P_{n},Q_{n},P_{n}^{-1},$ and $Q_{n}^{-1}$ are similar to
each other and thus%
\begin{align}
\lambda_{\max}(P_{n})  &  =\lambda_{\max}(Q_{n})=\lambda_{\min}^{-1}%
(P_{n})=\lambda_{\max}(Q_{n}^{-1})\nonumber \\
&  =\lambda_{\min}^{-1}(Q_{n})=\lambda_{\min}^{-1}(Q_{n}^{-1})\triangleq
\mathit{\Lambda}. \label{kama}%
\end{align}

\item There holds
\begin{equation}
cQ(\gamma)c^{\mathrm{T}}=n\gamma. \label{CQC}%
\end{equation}

\end{enumerate}
\end{lemma}

\section{\label{sec3}State Feedback}

A solution to Problem \ref{prob2} can be stated as follows.

\begin{theorem}
\label{the1}Let $T>0$ be a prescribed time, $\beta=\frac{2n+\delta
_{\mathrm{c}}}{n}+2d\mathit{\Lambda}$, and $\gamma_{0}=1/T.$ Design the
smooth time-varying state feedback%
\begin{equation}
\left \{
\begin{array}
[c]{rl}%
u_{0}(  t)  & =-\frac{3}{T-t}x_{0}(  t)-\frac{\beta}{2}(T-t),\\[1.5mm]
u(  t)  & =-\beta b^{\mathrm{T}}P(  \gamma (t)
)  L_{n}\left(  \frac{1}{T-t}\right)  x(  t),\\[1.5mm]
\gamma (  t)  & =\frac{T}{T-t}\gamma_{0}.
\end{array}
\right.  \label{u_0_u}%
\end{equation}
Then, for any initial condition and all $t\in \lbrack0,T)$, the states and
controls of the closed-loop system consisting of (\ref{sys}) and
(\ref{u_0_u}) satisfy%
\begin{align}
\left \vert x_{0}(t)\right \vert  &  \leq (  T-t)  ^{2}\left(
\frac{T-t}{T^{3}}\left \vert x_{0}(0)\right \vert +\frac{\beta t}{2T}\right)
,\label{A}\\
\left \vert u_{0}(t)\right \vert  &  \leq(T-t)\left(  3(  T-t)
\left(  \frac{\vert x_{0}(0) \vert }{T^{3}}+\frac{\beta}{2T}\right)
+\beta \right)  ,\label{B}\\
\left \Vert x\left(  t\right)  \right \Vert  &  \leq \upsilon_{1}\left(
T-t\right)  ^{\frac{3}{2}}\mathrm{e}^{\upsilon_{10}\left \vert x_{0}%
(0)\right \vert t}\left \Vert x(  0)  \right \Vert ,\label{C}\\
\Vert u(t) \Vert  &  \leq \upsilon_{2}(
T-t)  ^{\frac{1}{2}}\mathrm{e}^{\upsilon_{20}\vert x_{0}%
(0)\vert t}\left \Vert x\left(  0\right)  \right \Vert , \label{D}%
\end{align}
where $\upsilon_{1}$, $\upsilon_{10}$, $\upsilon_{2}$ and $\upsilon_{20}$ are
some positive constants, namely, Problem \ref{prob2} is solved.
\end{theorem}

\begin{proof}
Consider the $x_{0}$-subsystem in system (\ref{sys}). The closed-loop system
consisting of (\ref{sys}) and the first equation in (\ref{u_0_u}) can be
written as
\[
\dot{x}_{0}=-\frac{3}{T-t}x_{0}-\frac{\beta}{2}(T-t),
\]
whose solution can be given as%
\begin{align}
x_{0}(t)=  &  \mathrm{\exp}\left(  -\int_{0}^{t}\frac{3}{T-s}\mathrm{d}%
s\right)  x_{0}(0)\nonumber \\
&  -\frac{\beta}{2}\int_{0}^{t}\exp \left(  -\int_{s}^{t}\frac{3}{T-\tau
}\mathrm{d}\tau \right)  (T-s)\mathrm{d}s\nonumber \\
=  &  \left(  \frac{T-t}{T}\right)  ^{3}x_{0}(0)-\frac{\beta}{2}%
(T-t)^{2}+\frac{\beta}{2T}(T-t)^{3}\nonumber \\
=  &  (T-t)^{2}\left(  \frac{T-t}{T^{3}}x_{0}(0)-\frac{\beta t}{2T}\right)  ,
\label{x_0e}%
\end{align}
which proves (\ref{A}). Then the control $u_{0}$ can be written as the
open-loop form%
\begin{align}
u_{0}(t)=  &  -\frac{3}{T-t}x_{0}-\frac{\beta}{2}(T-t)\nonumber \\
=  &  -\frac{3}{T-t}\left(  \left(  T-t\right)  ^{2}\left(  \frac{T-t}{T^{3}%
}x_{0}(0)-\frac{\beta t}{2T}\right)  \right) \nonumber \\
&  -\frac{\beta}{2}(T-t)\nonumber \\
=  &  (T-t)\left(  3\left(  T-t\right)  \left(  -\frac{x_{0}(0)}{T^{3}}%
-\frac{\beta}{2T}\right)  +\beta \right) \nonumber \\
\triangleq &  (T-t)\left(  \theta(t)+\beta \right)  , \label{u0}%
\end{align}
which proves (\ref{B}).

Consider the $x$-subsystem in system (\ref{sys}). By using (\ref{z_i_x_i}),
system (\ref{sys}) can be written as%
\[
\left \{
\begin{array}
[c]{l}%
\dot{z}_{i}=\frac{n-i}{T-t}z_{i}+(  \beta+\theta(t))
z_{i+1}+(  T-t)  ^{i-n}\phi_{i},\\
\dot{z}_{n}=u+\phi_{n},
\end{array}
\right.
\]
where $i=1,2,\ldots,n-1$, which can be re-expressed as%
\[
\dot{z}=\frac{1}{T-t}A_{0}z+\theta(t)Az+\beta Az+\varphi+bu,
\]
where $\varphi=\varphi(t,u,x)=[\varphi_{1}(t,u,x),\varphi_{2}(t,u,x),\ldots
,\varphi_{n}$ $(t,u,x)]^{\mathrm{T}}$, and%
\[
A_{0}=\mathrm{diag}\{n-1,n-2,\ldots,2,1\}.
\]
Such a system can be further expressed as%
\begin{equation}
\dot{z}=\beta \left(  A-bb^{\mathrm{T}}P\right)  z+\theta(t)Az+\psi
,\label{z_dot}%
\end{equation}
where $\psi=\psi(t,u,z)=[\psi_{1}(t,u,z),\psi_{2}(t,u,z),\ldots,\psi_{n}$
$(t,u,z)]^{\mathrm{T}}$. Choose the Lyapunov-like function%
\[
V\left(  t,z(t)\right)  =\gamma(t) z^{\mathrm{T}}(t)P(\gamma(t))z(t),
\]
whose time-derivative along the trajectory of the closed-loop system
(\ref{z_dot}) can be written as%
\begin{align}
\dot{V}\left(  t,z\right)  = &  \dot{\gamma}z^{\mathrm{T}}Pz+\gamma \dot
{\gamma}z^{\mathrm{T}}\frac{\mathrm{d}P}{\mathrm{d}\gamma}z\nonumber \\
&  +\beta \gamma z^{\mathrm{T}}\left(  A^{\mathrm{T}}P+PA-2Pbb^{\mathrm{T}%
}P\right)  z\nonumber \\
&  +2\gamma \theta z^{\mathrm{T}}PAz+2\gamma z^{\mathrm{T}}P\psi \nonumber \\
\leq &  \dot{\gamma}z^{\mathrm{T}}Pz+\frac{\delta_{\mathrm{c}}}{n}\dot{\gamma
}z^{\mathrm{T}}Pz-\beta \gamma^{2}z^{\mathrm{T}}Pz\nonumber \\
&  +2\gamma \theta z^{\mathrm{T}}PAz+2\gamma z^{\mathrm{T}}P\psi,\label{v_dot}%
\end{align}
where we have used (\ref{ple1}) and (\ref{dP_bounded}). By using the Young's
inequalities with $k_{0}>0$ and $k_{1}>0$, we have%
\begin{align*}
2z^{\mathrm{T}}PAz &  \leq k_{0}\gamma z^{\mathrm{T}}Pz+\frac{1}{k_{0}\gamma
}z^{\mathrm{T}}A^{\mathrm{T}}PAz,\\
2z^{\mathrm{T}}P\psi &  \leq k_{1}\gamma z^{\mathrm{T}}Pz+\frac{1}{k_{1}%
\gamma}\psi^{\mathrm{T}}P\psi.
\end{align*}
In addition, it can be obtained from Lemma \ref{lem1}, (\ref{P}),
and (\ref{kama}) that%
\begin{align*}
\psi^{\mathrm{T}}P\psi &  =\psi^{\mathrm{T}}\gamma L_{n}P_{n}L_{n}\psi
\leq \mathit{\Lambda}\gamma \left(  L_{n}\psi \right)  ^{\mathrm{T}}\left(
L_{n}\psi \right)  \\
&  =d^{2}\mathit{\Lambda}\gamma^{3}\left(  L_{n}z\right)  ^{\mathrm{T}}\left(
L_{n}z\right)  \leq d^{2}\mathit{\Lambda}^{2}\gamma^{2}z^{\mathrm{T}}Pz.
\end{align*}
With this and (\ref{hh}) in Lemma \ref{lem1}, (\ref{v_dot}) can be continued as%
\begin{align*}
\dot{V}\left(  t,z\right)  \leq &  \dot{\gamma}z^{\mathrm{T}}Pz+\frac
{\delta_{\mathrm{c}}}{n}\dot{\gamma}z^{\mathrm{T}}Pz-\beta \gamma
^{2}z^{\mathrm{T}}Pz\\
&  +k_{0}\left \vert \theta \right \vert \gamma^{2}z^{\mathrm{T}}Pz+\frac
{\left \vert \theta \right \vert }{k_{0}}z^{\mathrm{T}}A^{\mathrm{T}}PAz\\
&  +k_{1}\gamma^{2}z^{\mathrm{T}}Pz+\frac{1}{k_{1}}\psi^{\mathrm{T}}P\psi \\
\leq &  \dot{\gamma}z^{\mathrm{T}}Pz+\frac{\delta_{\mathrm{c}}}{n}\dot{\gamma
}z^{\mathrm{T}}Pz-\beta \gamma^{2}z^{\mathrm{T}}Pz\\
&  +k_{0}\left \vert \theta \right \vert \gamma^{2}z^{\mathrm{T}}Pz+\frac
{3n^{2}\left \vert \theta \right \vert }{k_{0}}\gamma^{2}z^{\mathrm{T}}Pz\\
&  +k_{1}\gamma^{2}z^{\mathrm{T}}Pz+\frac{1}{k_{1}}d^{2}\mathit{\Lambda}%
^{2}\gamma^{2}z^{\mathrm{T}}Pz\\
= &  \left(  \dot{\gamma}+\frac{\delta_{\mathrm{c}}}{n}\dot{\gamma}%
-\beta \gamma^{2}+k_{1}\gamma^{2}+\frac{1}{k_{1}}d^{2}\mathit{\Lambda}%
^{2}\gamma^{2}\right)  z^{\mathrm{T}}Pz\\
&  +\left(  k_{0}\left \vert \theta \right \vert \gamma^{2}+\frac{3n^{2}}{k_{0}%
}\left \vert \theta \right \vert \gamma^{2}\right)  z^{\mathrm{T}}Pz\\
= &  \left(  1+\frac{\delta_{\mathrm{c}}}{n}\right)  \mu_{z}z^{\mathrm{T}%
}Pz+2\left \vert \theta \right \vert \sqrt{3}n\gamma^{2}z^{\mathrm{T}}Pz\\
= &  -\gamma^{2}z^{\mathrm{T}}Pz+\theta_{0}\gamma z^{\mathrm{T}}Pz,
\end{align*}
where%
\begin{align*}
\theta_{0} &  =6\left(  \frac{\left \vert x_{0}(0)\right \vert }{T^{3}}%
+\frac{\beta}{2T}\right)  \sqrt{3}n,\\
\mu_{z} &  =\dot{\gamma}-\frac{n}{n+\delta_{\mathrm{c}}}\beta \gamma^{2}%
+\frac{2nd\mathit{\Lambda}}{n+\delta_{\mathrm{c}}}\gamma^{2}=-\frac
{n}{n+\delta_{\mathrm{c}}}\gamma^{2},
\end{align*}
and we have taken $k_{0}=\sqrt{3}n$ and $k_{1}=d\mathit{\Lambda}$. By using
the comparison lemma in \cite{Khalil}, $V(t,z(t))$ satisfies
\begin{align}
V(t,z(t))  & \leq \exp \left(  -\int_{0}^{t}\gamma(s)\mathrm{d}s\right)
\mathrm{e}^{\theta_{0}t}V(0,z(0))\nonumber \\
& =(  T-t)  \mathrm{e}^{\theta_{0}t}V(0,z(0)),\quad t\in
\lbrack0,T).\label{V_t}%
\end{align}
It follows from (\ref{P}) that%
\[
V(t,z(t))=\gamma z^{\mathrm{T}}Pz\geq \gamma^{2}\lambda_{\min}(L_{n}(\gamma
_{0})P_{n}L_{n}(\gamma_{0}))\Vert z(t)\Vert^{2}.
\]
Therefore, we have%
\[
\left \Vert z\left(  t\right)  \right \Vert ^{2}\leq \frac{\left(  T-t\right)
^{3}\mathrm{e}^{\theta_{0}t}V\left(  0,z\left(  0\right)  \right)  }%
{\lambda_{\min}\left(  L_{n}\left(  \gamma_{0}\right)  P_{n}L_{n}\left(
\gamma_{0}\right)  \right)  },
\]
which, together with (\ref{z_i_x_i}), indicates that%
\begin{align}
\left \Vert x\left(  t\right)  \right \Vert  &  \leq \left \Vert L_{n}\left(
T-t\right)  \right \Vert \left \Vert z\left(  t\right)  \right \Vert \nonumber \\
&  \leq \left \Vert L_{n}\left(  T\right)  \right \Vert \frac{\left(  T-t\right)
^{\frac{3}{2}}\mathrm{e}^{\frac{1}{2}\theta_{0}t}\sqrt{V\left(  0,z\left(
0\right)  \right)  }}{\sqrt{\lambda_{\min}\left(  L_{n}\left(  \gamma
_{0}\right)  P_{n}L_{n}\left(  \gamma_{0}\right)  \right)  }}.\label{xi}%
\end{align}
In addition, it can be obtained from (\ref{z_i_x_i}) and (\ref{P}) that%
\begin{align}
&  V\left(  0,z\left(  0\right)  \right)  =\gamma_{0}z^{\mathrm{T}}(0)P\left(
\gamma_{0}\right)  z(0)\nonumber \\
&  \leq \gamma_{0}^{2}\lambda_{\max}\left(  L_{n}\left(  \gamma_{0}\right)
P_{n}L_{n}\left(  \gamma_{0}\right)  \right)  \left \Vert z(0)\right \Vert
^{2}\nonumber \\
&  \leq \gamma_{0}^{2}\lambda_{\max}\left(  L_{n}\left(  \gamma_{0}\right)
P_{n}L_{n}\left(  \gamma_{0}\right)  \right)  \left \Vert L_{n}\left(  \frac
{1}{T}\right)  \right \Vert ^{2}\left \Vert x(0)\right \Vert ^{2},\label{V_0}%
\end{align}
combining which with (\ref{xi}), indicates that%
\begin{align*}
\left \Vert x\left(  t\right)  \right \Vert  &  \leq \left \Vert L_{n}\left(
T\right)  \right \Vert \left(  T-t\right)  ^{\frac{3}{2}}\frac{\mathrm{e}%
^{\frac{1}{2}\theta_{0}t}\sqrt{V\left(  0,z\left(  0\right)  \right)  }}%
{\sqrt{\lambda_{\min}\left(  L_{n}\left(  \gamma_{0}\right)  P_{n}L_{n}\left(
\gamma_{0}\right)  \right)  }}\\
\leq &  \gamma_{0}\sqrt{\omega_{1}}\left \Vert L_{n}\left(  T\right)
\right \Vert \left(  T-t\right)  ^{\frac{3}{2}}\mathrm{e}^{\frac{1}{2}%
\theta_{0}t}\left \Vert L_{n}\left(  \frac{1}{T}\right)  \right \Vert \left \Vert
x(0)\right \Vert ,
\end{align*}
where $\omega_{1}=\lambda_{\max}(L_{n}(\gamma_{0})P_{n}L_{n}(\gamma
_{0}))/\lambda_{\min}(L_{n}(\gamma_{0})P_{n}L_{n}$ $(\gamma_{0}))$. Thus we
have proven (\ref{C}).

Finally, it follows from (\ref{bb}), (\ref{V_t}), (\ref{V_0}) that%
\begin{align*}
\Vert u(t)\Vert^{2}  &  =\beta^{2}z^{\mathrm{T}}P(\gamma(t))bb^{\mathrm{T}%
}P(\gamma(t))z\\
&  \leq \beta^{2}b^{\mathrm{T}}P(\gamma(t))bz^{T}P(\gamma(t))z\\
&  =\beta^{2}n\gamma z^{T}P(\gamma(t))z
 =\beta^{2}nV(t,z(t))\\
&  \leq \beta^{2}n(T-t)\mathrm{e}^{\theta_{0}t}V(0,z(0))\\
&  =\beta^{2}n(T-t)\mathrm{e}^{\theta_{0}t}\lambda_{\max}(L_{n}(\gamma
_{0})P_{n}L_{n}(\gamma_{0}))\\
&  \times \left \Vert L_{n}\left(  \frac{1}{T}\right)  \right \Vert ^{2}\Vert
x(0)\Vert^{2},
\end{align*}
which proves (\ref{D}). The proof is finished.
\end{proof}

If the $x_{0}$-subsystem in system (\ref{sys}) is of the following form
\cite{Zhang20tac}%
\begin{equation}
\dot{x}_{0}=u_{0}+c_{0}x_{0}, \label{x0}%
\end{equation}
where $c_{0}$ is a known constant, we can get the following result.

\begin{proposition}
Let $T>0$ be a prescribed time, $\beta=(2n+\delta_{\mathrm{c}})\mathrm{e}%
^{\left \vert c_{0}\right \vert T}/n+2d\mathit{\Lambda}\mathrm{e}^{\left \vert
c_{0}\right \vert T}$, and $\gamma_{0}=1/T$. Then, for any initial condition
and all $t\in \lbrack0,T)$, the state and control of the closed-loop system
consisting of (\ref{sys}) with $x_{0}$-subsystem be defined in (\ref{x0}) and
the following smooth time-varying feedback
\begin{equation}
\left \{
\begin{array}
[c]{rl}%
u_{0}\left(  t\right)  & =-\frac{3}{T-t}x_{0}-\frac{\beta}{2}\mathrm{e}%
^{c_{0}t}(T-t),\\[1.5mm]
u\left(  t\right)  & =-\beta \mathrm{e}^{c_{0}t}b^{\mathrm{T}}P\left(
\gamma \left(  t\right)  \right)  L_{n}\left(  \frac{1}{T-t}\right)  x,\\[1.5mm]
\gamma \left(  t\right)  & =\frac{T}{T-t}\gamma_{0},
\end{array}
\right.  \label{u_0_u1}%
\end{equation}
satisfy%
\begin{align*}
\left \vert x_{0}(t)\right \vert  &  \leq \left(  T-t\right)  ^{2}\mathrm{e}%
^{c_{0}t}\left(  \frac{\left(  T-t\right)  \left \vert x_{0}(0)\right \vert
}{T^{3}}+\frac{\beta t}{2T}\right)  ,\\
\left \vert u_{0}(t)\right \vert  &  \leq \left(  T-t\right)  \mathrm{e}^{c_{0}%
t}\left(  3\left(  T-t\right)  \left(  \frac{\left \vert x_{0}(0)\right \vert
}{T^{3}}+\frac{\beta}{2T}\right)  +\beta \right)  ,\\
\left \Vert x\left(  t\right)  \right \Vert  &  \leq \upsilon_{11}\left(
T-t\right)  ^{\frac{3}{2}}\mathrm{e}^{\upsilon_{12}\left \vert x_{0}%
(0)\right \vert t}\left \Vert x\left(  0\right)  \right \Vert ,\\
\left \Vert u\left(  t\right)  \right \Vert  &  \leq \upsilon_{21}(T-t)^{\frac
{1}{2}}\mathrm{e}^{\upsilon_{22}\left \vert x_{0}(0)\right \vert t}\left \Vert
x\left(  0\right)  \right \Vert ,
\end{align*}
where $\upsilon_{11}$, $\upsilon_{12}$, $\upsilon_{21}$ and $\upsilon_{22}$
are some positive constants.
\end{proposition}

\begin{proof}
The closed-loop system consisting of (\ref{x0}) and $u_{0}\left(  t\right)  $
in the first equation in (\ref{u_0_u1}) can be written as
\[
\dot{x}_{0}(t)=-\frac{3}{T-t}x_{0}(t)+c_{0}x_{0}(t)-\frac{\beta}{2}\mathrm{e}^{c_{0}%
t}(T-t),
\]
whose solution can be given as%
\begin{align*}
x_{0}(t)=  &  \mathrm{e}^{c_{0}t}\mathrm{\exp}\left(  -\int_{0}^{t}\frac
{3}{T-s}\mathrm{d}s\right)  x_{0}(0)\\
&  -\frac{\beta}{2}\int_{0}^{t}\mathrm{e}^{c_{0}t}\exp \left(  -\int_{s}%
^{t}\frac{3}{T-\tau}\mathrm{d}\tau \right)  (T-s)\mathrm{d}s\\
=  &  \left(  \frac{T-t}{T}\right)  ^{3}x_{0}(0)\mathrm{e}^{c_{0}t}%
-\frac{\beta}{2}\mathrm{e}^{c_{0}t}\left(  T-t\right)  ^{2}\\
&  +\frac{\beta \mathrm{e}^{c_{0}t}\left(  T-t\right)  ^{3}}{2T}\\
=  &  \left(  T-t\right)  ^{2}\left(  \frac{\left(  T-t\right)  x_{0}%
(0)}{T^{3}}\mathrm{e}^{c_{0}t}-\frac{t\beta}{2T}\mathrm{e}^{c_{0}t}\right)  ,
\end{align*}
which proves the expression of $\left \vert x_{0}(t)\right \vert $ in this
proposition. Then the control $u_{0}$ can be written as the open-loop form%
\begin{align*}
&  u_{0}(t)=-\frac{3}{T-t}x_{0}-\frac{\beta}{2}\mathrm{e}^{c_{0}t}(T-t)\\
  =&-\frac{3}{T-t}\left(  \left(  T-t\right)  ^{2}\left(  \frac{\left(
T-t\right)  x_{0}(0)}{T^{3}}\mathrm{e}^{c_{0}t}-\frac{t\beta}{2T}%
\mathrm{e}^{c_{0}t}\right)  \right) \\
&  -\frac{\beta}{2}\mathrm{e}^{c_{0}t}(T-t)\\
  =&\left(  T-t\right)  \left(  -3\left(  T-t\right)  \left(  \frac{x_{0}%
(0)}{T^{3}}\mathrm{e}^{c_{0}t}+\frac{\beta \mathrm{e}^{c_{0}t}}{2T}\right)
+\beta \mathrm{e}^{c_{0}t}\right) \\
&  \triangleq(T-t)\left(  \theta(t)+\beta \mathrm{e}^{c_{0}t}\right)  ,
\end{align*}
which proves the expression of $\left \vert u_{0}(t)\right \vert $ in this
proposition. The rest of the proof is similar as the proof of Theorem
\ref{the1}, and is omitted for brevity.
\end{proof}

When $c_{0}=0$, the above theorem reduces to Theorem \ref{the1}. In addition,
similar to Remark 4 in \cite{zhou23auto}, the proposed method can also
handle the nonholonomic system (\ref{sys}) with the nonlinear functions
satisfying
\[
\left \vert \phi_{i}\left(  t,u,x\right)  \right \vert \leq \varphi
(x_{0})(\left \vert x_{1}\right \vert +\left \vert x_{2}\right \vert
+\cdots+\left \vert x_{i}\right \vert )
\]
in which $\varphi(x_{0})$ are some known positive continuous functions. Due to
space limitations, the details are not given here.

\section{\label{sec3_1}Observed-Based Output Feedback}

To state a solution to Problem \ref{prob3}, we first denote the parameter
$\beta$ as
\begin{equation}
\beta=\max \left \{  \beta_{1},\beta_{2}\right \}  , \label{beta}%
\end{equation}
where
\begin{align*}
\beta_{1}  &  =8\sqrt{2}nd\mathit{\Lambda}cQ_{n}P_{n}Q_{n}c^{\mathrm{T}%
}+2\left(  2+\frac{\delta_{\mathrm{c}}}{n}\right)  ,\\
\beta_{2}  &  =4\sqrt{2}d\mathit{\Lambda}+2\left(  2n+2-\frac{1}{n}\right)  ,
\end{align*}
in which $d$ is defined in (\ref{d}).

\begin{theorem}
Let $T>0$ be a prescribed time and $\gamma_{0}=1/T$. Design the smooth
time-varying observed-based output feedback
\begin{equation}
\left \{
\begin{array}
[c]{l}%
\dot{\xi}(t)=\beta A\xi(t)+bu(t)+\beta Q(\gamma)c^{\mathrm{T}}\left(
\gamma^{n-1}y_{2}-c\xi(t)\right)  ,\\[1mm]
u(t)=-\beta b^{\mathrm{T}}P(\gamma)\xi(t),\\[1mm]
\gamma=\gamma(t)=\frac{T}{T-t}\gamma_{0}.
\end{array}
\right.  \label{obser_feed}%
\end{equation}
Then, for any initial condition and all $t\in \lbrack0,T)$, the states
$x_{0}(t)$, $x(t)$, $\xi \left(  t\right)  $, and controls $u_{0}(t)$,
$u\left(  t\right)  $ of the closed-loop system consisting of (\ref{sys}) and
(\ref{obser_feed}) satisfy, (\ref{A}), (\ref{B}) and
\begin{align}
\left \Vert x\left(  t\right)  \right \Vert  &  \leq \upsilon_{3}\mathrm{e}%
^{\upsilon_{30}\left \vert x_{0}\right \vert t}\left(  T-t\right)  ^{\frac{3}%
{2}}\left(  \left \Vert x\left(  0\right)  \right \Vert +\left \Vert \xi \left(
0\right)  \right \Vert \right)  ,\label{E}\\
\left \Vert \xi \left(  t\right)  \right \Vert  &  \leq \upsilon_{4}%
\mathrm{e}^{\upsilon_{40}\left \vert x_{0}\right \vert t}\left(  T-t\right)
^{\frac{3}{2}}\left(  \left \Vert x\left(  0\right)  \right \Vert +\left \Vert
\xi \left(  0\right)  \right \Vert \right)  ,\label{F}\\
\left \Vert u\left(  t\right)  \right \Vert  &  \leq \upsilon_{5}\mathrm{e}%
^{\upsilon_{50}\left \vert x_{0}\right \vert t}\left(  T-t\right)  ^{\frac{1}%
{2}}\left(  \left \Vert x\left(  0\right)  \right \Vert +\left \Vert \xi \left(
0\right)  \right \Vert \right)  , \label{G}%
\end{align}
where $\upsilon_{3}$, $\upsilon_{30}$, $\upsilon_{4}$, $\upsilon_{40}$,
$\upsilon_{5}$ and $\upsilon_{50}$ are some positive constants, namely,
Problem \ref{prob3} is solved.
\end{theorem}

\begin{proof}
In this proof, if not specified, we drop the dependence of variables on $t$,
for example, we use $P=P(\gamma(t))$ and $Q=Q(\gamma(t))$. Consider the
$x$-subsystem in (\ref{sys}). Denote $e=z-\xi$. By using (\ref{z_i_x_i}), the
closed-loop system can be written as%
\begin{align}
\dot{\xi} &  =\beta (  A-bb^{\mathrm{T}}P)  \xi+\beta Qc^{\mathrm{T}%
}ce,\label{xi_clo}\\
\dot{e} &  =\beta Ae+\theta Az+\psi-\beta Qc^{\mathrm{T}}ce.\label{e_clo}%
\end{align}
Choose the Lyapunov-like function%
\[
V_{\xi}(  t,\xi(t))  =\gamma(t) \xi^{\mathrm{T}}(t)P(\gamma(t))\xi(t),
\]
whose time-derivative along system (\ref{xi_clo}) can be written as%
\begin{align}
\dot{V}_{\xi}(  t,\xi )  = &  \dot{\gamma}\xi^{\mathrm{T}}%
P\xi+\gamma \dot{\gamma}\xi^{\mathrm{T}}\frac{\mathrm{d}P}{\mathrm{d}\gamma}%
\xi+2\beta \gamma \xi^{\mathrm{T}}PQc^{\mathrm{T}}ce\nonumber \\
&  +\beta \gamma \xi^{\mathrm{T}}(  A^{\mathrm{T}}P+PA-2Pbb^{\mathrm{T}%
}P)  \xi \nonumber \\
\leq &  \dot{\gamma}\xi^{\mathrm{T}}P\xi+\frac{\delta_{\mathrm{c}}\dot{\gamma
}}{n}\xi^{\mathrm{T}}P\xi \nonumber \\
&  -\beta \gamma^{2}\xi^{\mathrm{T}}P\xi+2\beta \gamma \xi^{\mathrm{T}%
}PQc^{\mathrm{T}}ce,\label{V_dot_xi}%
\end{align}
where we have used (\ref{ple1}) and (\ref{dP_bounded}). By using Young's
inequality with $k_{2}>0$, it follows that%
\[
2\beta \xi^{\mathrm{T}}PQc^{\mathrm{T}}ce\leq k_{2}\beta \gamma \xi^{\mathrm{T}%
}P\xi+\frac{\beta}{k_{2}\gamma}e^{\mathrm{T}}c^{\mathrm{T}}cQPQc^{\mathrm{T}%
}ce.
\]
In addition, according to Theorem 2 in \cite{zhou21tac}, we have
$c^{\mathrm{T}}cQPQc^{\mathrm{T}}c\leq \gamma^{2n+2}nk_{3}Q^{-1}$,
where $k_{3}=cQ_{n}P_{n}$ $Q_{n}c^{\mathrm{T}}$.

Therefore, (\ref{V_dot_xi}) can be continued as
\begin{align}
\dot{V}_{\xi}\left(  t,\xi \right)  \leq &  \dot{\gamma}\xi^{\mathrm{T}}%
P\xi+\frac{\delta_{\mathrm{c}}\dot{\gamma}}{n}\xi^{\mathrm{T}}P\xi-\beta
\gamma^{2}\xi^{\mathrm{T}}P\xi+k_{2}\beta \gamma^{2}\xi^{\mathrm{T}}%
P\xi \nonumber \\
&  +\frac{\beta}{k_{2}}\gamma^{2n+2}nk_{3}e^{\mathrm{T}}Q^{-1}e\nonumber \\
=  &  \left(  \frac{\dot{\gamma}}{\gamma}+\frac{\delta_{\mathrm{c}}\dot
{\gamma}}{n\gamma}-\beta \gamma+k_{2}\beta \gamma \right)  \gamma \xi^{\mathrm{T}%
}P\xi \nonumber \\
&  +\frac{\beta}{k_{2}}nk_{3}\gamma^{2n+2}e^{\mathrm{T}}Q^{-1}e. \label{V_xi}%
\end{align}
Take another Lyapunov-like function%
\[
V_{e}(  t,e(t))  =\gamma^{2n+1}(t)e^{\mathrm{T}}(t)Q^{-1}(\gamma(t))e(t),
\]
whose time-derivative along system (\ref{e_clo}) can be written as
\begin{align}
\dot{V}_{e}(t,e)=  &  (2n+1)\gamma^{2n}\dot{\gamma}e^{\mathrm{T}}%
Q^{-1}e+\gamma^{2n+1}\dot{\gamma}e^{\mathrm{T}}\frac{\mathrm{d}Q^{-1}%
}{\mathrm{d}\gamma}e\nonumber \\
&  +\gamma^{2n+1}(\beta Ae+\theta Az+\psi-\beta Qc^{\mathrm{T}}ce)^{\mathrm{T}%
}Q^{-1}e\nonumber \\
&  +\gamma^{2n+1}e^{\mathrm{T}}Q^{-1}(\beta Ae+\theta Az+\psi-Qc^{\mathrm{T}%
}ce)\nonumber \\
\leq &  (2n+1)\gamma^{2n}\dot{\gamma}e^{\mathrm{T}}Q^{-1}e-\frac{\dot{\gamma}%
}{n}\gamma^{2n}e^{\mathrm{T}}Q^{-1}e\nonumber \\
&  +\gamma^{2n+1}(\beta Ae+\theta Az+\psi-\beta Qc^{\mathrm{T}}ce)^{\mathrm{T}%
}Q^{-1}e\nonumber \\
&  +\gamma^{2n+1}e^{\mathrm{T}}Q^{-1}(\beta Ae+\theta Az+\psi-\beta
Qc^{\mathrm{T}}ce)\nonumber \\
\leq &  \left(  2n+1-\frac{1}{n}\right)  \gamma^{2n}\dot{\gamma}e^{\mathrm{T}%
}Q^{-1}e+2\gamma^{2n+1}e^{\mathrm{T}}Q^{-1}\psi \nonumber \\
&  +\gamma^{2n+1}\beta e^{\mathrm{T}}Q^{-1}(QA^{\mathrm{T}}+AQ-2Qc^{\mathrm{T}%
}cQ)Q^{-1}e\nonumber \\
&  +2\theta \gamma^{2n+1}e^{\mathrm{T}}Q^{-1}Az, \label{V_e_dot}%
\end{align}
where we have used (\ref{eqdq}). By using Young's inequalities with $k_{4}>0$
and $k_{5}>0$, we can get%
\begin{align*}
2z^{\mathrm{T}}A^{\mathrm{T}}Q^{-1}e  &  \leq \frac{k_{4}}{\gamma}%
z^{\mathrm{T}}A^{\mathrm{T}}Q^{-1}Az+\frac{\gamma}{k_{4}}e^{\mathrm{T}}%
Q^{-1}e,\\
2e^{\mathrm{T}}Q^{-1}\psi &  \leq \frac{k_{5}}{\gamma}\psi^{\mathrm{T}}%
Q^{-1}\psi+\frac{\gamma}{k_{5}}e^{\mathrm{T}}Q^{-1}e.
\end{align*}
On the one hand, it follows from Lemma \ref{lem1}, (\ref{Q}) and (\ref{kama})
that
\begin{align*}
\psi^{\mathrm{T}}Q^{-1}\psi &  =\psi^{\mathrm{T}}\left(  \gamma^{2n-1}%
L_{n}^{-1}Q_{n}L_{n}^{-1}\right)  ^{-1}\psi \\
&  \leq \frac{\mathit{\Lambda}}{\gamma^{2n-1}}(  L_{n}\psi)
^{\mathrm{T}}\left(  L_{n}\psi \right) \\
&  \leq \frac{d^{2}\mathit{\Lambda}}{\gamma^{2n-3}}(  L_{n}z)
^{\mathrm{T}}(  L_{n}z)  .
\end{align*}
On the other hand, in light of $z=\xi+e$, (\ref{P}), (\ref{Q}) and
(\ref{kama}), we can get%
\begin{align*}
\left(  L_{n}z\right)  ^{\mathrm{T}}\left(  L_{n}z\right)   &  =\left(
L_{n}\left(  \xi+e\right)  \right)  ^{\mathrm{T}}\left(  L_{n}\left(
\xi+e\right)  \right) \\
&  \leq2\left(  L_{n}\xi \right)  ^{\mathrm{T}}\left(  L_{n}\xi \right)
+2\left(  L_{n}e\right)  ^{\mathrm{T}}\left(  L_{n}e\right) \\
&  \leq2\mathit{\Lambda}\left(  \frac{\xi^{\mathrm{T}}P\xi}{\gamma}%
+\gamma^{2n-1}e^{\mathrm{T}}Q^{-1}e\right)  .
\end{align*}
Therefore, we have
\begin{align*}
\psi^{\mathrm{T}}Q^{-1}\psi &  \leq \frac{2d^{2}\mathit{\Lambda}^{2}}%
{\gamma^{2n-3}}\left(  \frac{\xi^{\mathrm{T}}P\xi}{\gamma}+\gamma
^{2n-1}e^{\mathrm{T}}Q^{-1}e\right) \\
&  =\frac{2d^{2}\mathit{\Lambda}^{2}}{\gamma^{2n-2}}\xi^{\mathrm{T}}%
P\xi+2d^{2}\mathit{\Lambda}^{2}\gamma^{2}e^{\mathrm{T}}Q^{-1}e.
\end{align*}
In view of $L_{n}A=AL_{n}\gamma$, (\ref{P}) and (\ref{kama}), it yields that
\begin{align*}
A^{\mathrm{T}}Q^{-1}A  &  =A^{\mathrm{T}}\frac{1}{\gamma^{2n-1}}L_{n}%
Q_{n}^{^{-1}}L_{n}A\\
&  =\frac{1}{\gamma^{2n-3}}L_{n}A^{\mathrm{T}}Q_{n}^{^{-1}}AL_{n}\\
&  \leq \lambda_{\max}\left(  A^{\mathrm{T}}Q_{n}^{^{-1}}A\right)  \frac
{1}{\gamma^{2n-3}}L_{n}L_{n}\\
&  =\lambda_{\max}\left(  A^{\mathrm{T}}Q_{n}^{^{-1}}A\right)  \mathit{\Lambda
}\frac{1}{\gamma^{2n-3}}L_{n}\mathit{\Lambda}^{-1}L_{n}\\
&  \leq \lambda_{\max}\left(  A^{\mathrm{T}}Q_{n}^{^{-1}}A\right)
\mathit{\Lambda}\frac{1}{\gamma^{2n-3}}L_{n}Q_{n}^{^{-1}}L_{n}\\
&  =\lambda_{\max}\left(  A^{\mathrm{T}}Q_{n}^{^{-1}}A\right)  \mathit{\Lambda
}\gamma^{2}Q^{-1}
\triangleq k_{8}\gamma^{2}Q^{-1}.
\end{align*}
Again in view of $z=\xi+e$, (\ref{P}), and (\ref{kama}), it follows that
\begin{align*}
z^{\mathrm{T}}Q^{-1}z  &  =\left(  \xi+e\right)  ^{\mathrm{T}}Q^{-1}\left(
\xi+e\right) \\
&  \leq2\xi^{\mathrm{T}}Q^{-1}\xi+2e^{\mathrm{T}}Q^{-1}e\\
&  \leq \frac{2}{\gamma^{2n-1}}\xi^{\mathrm{T}}L_{n}Q_{n}^{^{-1}}L_{n}%
\xi+2e^{\mathrm{T}}Q^{-1}e\\
&  \leq \frac{2\mathit{\Lambda}}{\gamma^{2n-1}}\xi^{\mathrm{T}}L_{n}L_{n}%
\xi+2e^{\mathrm{T}}Q^{-1}e\\
&  \leq \frac{2\mathit{\Lambda}^{2}}{\gamma^{2n-1}}\xi^{\mathrm{T}}L_{n}%
P_{n}L_{n}\xi+2e^{\mathrm{T}}Q^{-1}e\\
&  =\frac{2\mathit{\Lambda}^{2}}{\gamma^{2n}}\xi^{\mathrm{T}}P\xi
+2e^{\mathrm{T}}Q^{-1}e.
\end{align*}
Then we have%
\begin{align*}
z^{\mathrm{T}}A^{\mathrm{T}}Q^{-1}Az  &  \leq k_{8}\gamma^{2}z^{\mathrm{T}%
}Q^{-1}z\\
&  \leq \frac{2\mathit{\Lambda}^{2}k_{8}}{\gamma^{2n-2}}\xi^{\mathrm{T}}%
P\xi+2k_{8}\gamma^{2}e^{\mathrm{T}}Q^{-1}e.
\end{align*}
With the above results and (\ref{ple2}), (\ref{V_e_dot}) can be written as%
\begin{align}
\dot{V}_{e}\left(  t,e\right)  \leq &  \left(  2n+1-\frac{1}{n}\right)
\gamma^{2n}\dot{\gamma}e^{\mathrm{T}}Q^{-1}e\nonumber \\
&  -\gamma^{2n+2}\beta e^{\mathrm{T}}Q^{-1}e+\left \vert \theta \right \vert
\gamma^{2n}k_{4}z^{\mathrm{T}}A^{\mathrm{T}}Q^{-1}Az\nonumber \\
&  +\left \vert \theta \right \vert \frac{1}{k_{4}}\gamma^{2n+2}e^{\mathrm{T}%
}Q^{-1}e+\gamma^{2n}k_{5}\psi^{\mathrm{T}}Q^{-1}\psi \nonumber \\
&  +\frac{1}{k_{5}}\gamma^{2n+2}e^{\mathrm{T}}Q^{-1}e\nonumber \\
\leq &  \left(  \left(  2n+1-\frac{1}{n}\right)  \frac{\dot{\gamma}}{\gamma
}-\beta \gamma \right)  \gamma^{2n+1}e^{\mathrm{T}}Q^{-1}e\nonumber \\
&  +\left(  2k_{5}d^{2}\mathit{\Lambda}^{2}+\frac{1}{k_{5}}\right)
\gamma^{2n+2}e^{\mathrm{T}}Q^{-1}e\nonumber \\
&  +2\left \vert \theta \right \vert \mathit{\Lambda}^{2}k_{4}k_{8}\gamma^{2}%
\xi^{\mathrm{T}}P\xi \nonumber \\
&  +\left \vert \theta \right \vert \left(  2k_{4}k_{8}+\frac{1}{k_{4}}\right)
\gamma^{2n+2}e^{\mathrm{T}}Q^{-1}e\nonumber \\
&  +2k_{5}d^{2}\mathit{\Lambda}^{2}\gamma^{2}\xi^{\mathrm{T}}P\xi. \label{V_e}%
\end{align}
Take the total Lyapunov-like function%
\[
\mathcal{V}=\mathcal{V}\left(  t,\xi,e\right)  =V_{\xi}\left(  t,\xi \right)
+k_{7}V_{e}\left(  t,e\right)  ,
\]
where $k_{7}$ is a positive constant to be designed later. It can be calculated from
(\ref{V_xi}) and (\ref{V_e}) that%
\begin{align*}
\mathcal{\dot{V}}=  &  \dot{V}_{\xi}\left(  t,\xi \right)  +k_{7}\dot{V}%
_{e}\left(  t,e\right) \\
\leq &  \left(  \dot{\gamma}+\frac{\delta_{\mathrm{c}}\dot{\gamma}}{n}%
-\beta \gamma^{2}+k_{2}\beta \gamma^{2}\right)  \xi^{\mathrm{T}}P\xi \\
&  +\frac{\beta}{k_{2}}\gamma^{2n+2}nk_{3}e^{\mathrm{T}}Q^{-1}e\\
&  +k_{7}\left(  \left(  2n+1-\frac{1}{n}\right)  \dot{\gamma}-\beta \gamma
^{2}\right)  \gamma^{2n}e^{\mathrm{T}}Q^{-1}e\\
&  +k_{7}\left(  2d^{2}\mathit{\Lambda}^{2}k_{5}+\frac{1}{k_{5}}\right)
\gamma^{2n+2}e^{\mathrm{T}}Q^{-1}e\\
&  +2\left \vert \theta \right \vert \mathit{\Lambda}^{2}k_{4}k_{7}k_{8}%
\gamma^{2}\xi^{\mathrm{T}}P\xi \\
&  +\left \vert \theta \right \vert \left(  2k_{4}k_{8}+\frac{1}{k_{4}}\right)
k_{7}\gamma^{2n+2}e^{\mathrm{T}}Q^{-1}e\\
&  +2d^{2}\mathit{\Lambda}^{2}k_{5}k_{7}\gamma^{2}\xi^{\mathrm{T}}P\xi \\
=  &  \mu_{\xi}\xi^{\mathrm{T}}P\xi+\mu_{e}k_{7}\gamma^{2n}e^{\mathrm{T}%
}Q^{-1}e+2\left \vert \theta \right \vert \mathit{\Lambda}^{2}k_{7}k_{4}%
k_{8}\gamma^{2}\xi^{\mathrm{T}}P\xi \\
&  +\frac{\left \vert \theta \right \vert }{k_{4}}k_{7}\gamma^{2n+2}%
e^{\mathrm{T}}Q^{-1}e+2k_{4}k_{8}\left \vert \theta \right \vert k_{7}%
\gamma^{2n+2}e^{\mathrm{T}}Q^{-1}e,
\end{align*}
where%
\begin{align*}
\mu_{\xi}=  &  \dot{\gamma}+\frac{\delta_{\mathrm{c}}\dot{\gamma}}{n}%
-\beta \gamma^{2}+k_{2}\beta \gamma^{2}+2d^{2}\mathit{\Lambda}^{2}k_{5}%
k_{7}\gamma^{2}\\
\mu_{e}=  &  \left(  2n+1-\frac{1}{n}\right)  \dot{\gamma}-\beta \gamma
^{2}+\frac{1}{k_{5}}\gamma^{2}\\
&  +2d^{2}\mathit{\Lambda}^{2}k_{5}\gamma^{2}+\frac{\beta nk_{3}}{k_{2}k_{7}%
}\gamma^{2}.
\end{align*}
Take the parameters as $k_{2}=1/2, k_{5}=1/(\sqrt{2}d\mathit{\Lambda})$ and
$k_{7}=4nk_{3}$. Then it follows from (\ref{beta}) that%
\begin{align*}
\mu_{\xi}=  &  \left(  1+\frac{\delta_{\mathrm{c}}}{n}\right)  \dot{\gamma
}-\left(  1-k_{2}\right)  \beta \gamma^{2}+2d^{2}\mathit{\Lambda}^{2}k_{5}%
k_{7}\gamma^{2}\\
\leq &  \left(  1+\frac{\delta_{\mathrm{c}}}{n}\right)  \dot{\gamma}-\frac
{1}{2}\beta_{1}\gamma^{2}+4\sqrt{2}d\mathit{\Lambda}cQ_{n}P_{n}Q_{n}%
c^{\mathrm{T}}\gamma^{2}\\
=  &  \left(  1+\frac{\delta_{\mathrm{c}}}{n}\right)  \dot{\gamma}-\left(
1+\frac{\delta_{\mathrm{c}}}{n}\right)  \gamma^{2}-\gamma^{2} = -\gamma^{2},\\
\mu_{e}=  &  \left(  2n+1-\frac{1}{n}\right)  \dot{\gamma}-\left(
1-\frac{nk_{3}}{k_{2}k_{7}}\right)  \beta \gamma^{2}+\frac{1}{k_{5}}\gamma
^{2}\\
&  +2d^{2}\mathit{\Lambda}^{2}k_{5}\gamma^{2}\\
\leq &  \left(  2n+1-\frac{1}{n}\right)  \dot{\gamma}-\frac{1}{2}\beta
_{2}\gamma^{2}+2\sqrt{2}d\mathit{\Lambda}\gamma^{2}\\
=  &  \left(  2n+1-\frac{1}{n}\right)  \dot{\gamma}-\left(  2n+1-\frac{1}%
{n}\right)  \gamma^{2}-\gamma^{2}
=   -\gamma^{2}.
\end{align*}
Therefore, we have%
\begin{align*}
\mathcal{\dot{V}}\left(  t,\xi,e\right)  \leq &  -\gamma \mathcal{V}\left(
t,\xi,e\right)  +2\left \vert \theta \right \vert \mathit{\Lambda}^{2}k_{7}%
k_{4}k_{8}\gamma^{2}\xi^{\mathrm{T}}P\xi \\
&  +\left \vert \theta \right \vert \frac{1}{k_{4}}k_{7}\gamma^{2n+2}%
e^{\mathrm{T}}Q^{-1}e\\
&  +2\left \vert \theta \right \vert k_{4}k_{8}k_{7}\gamma^{2n+2}e^{\mathrm{T}%
}Q^{-1}e\\
=  &  -\gamma \mathcal{V}\left(  t,\xi,e\right)  +\frac{6\left \vert
x_{0}(0)\right \vert }{T^{2}}\mathit{\Lambda}^{2}k_{7}k_{4}k_{8}\gamma
\xi^{\mathrm{T}}P\xi \\
&  +3\beta \mathit{\Lambda}^{2}k_{7}k_{4}k_{8}\gamma \xi^{\mathrm{T}}P\xi \\
&  +3\left(  \frac{\left \vert x_{0}(0)\right \vert }{T^{2}}+\frac{\beta}%
{2}\right)  \frac{1}{k_{4}}k_{7}\gamma^{2n+1}e^{\mathrm{T}}Q^{-1}e\\
&  +3\left(  \frac{2\left \vert x_{0}(0)\right \vert }{T^{2}}+\beta \right)
k_{4}k_{8}k_{7}\gamma^{2n+1}e^{\mathrm{T}}Q^{-1}e\\
\leq &  -\gamma \mathcal{V}\left(  t,\xi,e\right)  +\lambda \mathcal{V}\left(
t,\xi,e\right)  ,
\end{align*}
where
\begin{align*}
\lambda=  &  3\left(  \frac{2\left \vert x_{0}(0)\right \vert }{T^{2}}%
+\beta \right)  \mathit{\Lambda}^{2}k_{7}k_{4}k_{8}+3\left(  \frac{\left \vert
x_{0}(0)\right \vert }{T^{2}}+\frac{\beta}{2}\right)  \frac{1}{k_{4}}\\
&  +3\left(  \frac{2\left \vert x_{0}(0)\right \vert }{T^{2}}+\beta \right)
k_{4}k_{8},
\end{align*}
which, together with the comparison lemma in \cite{Khalil}, indicates
\begin{equation}
\mathcal{V}(t,\xi(t),e(t))\leq(T-t)\mathrm{e}^{\lambda t}\mathcal{V}(0,\xi
(0),e(0)),\quad t\in \lbrack0,T). \label{v_xi_e}%
\end{equation}
It can be obtained from (\ref{P}) and (\ref{Q}) that
\begin{align*}
V_{\xi}(  t,\xi )   &  =\gamma \xi^{\mathrm{T}}P(
\gamma)  \xi
\geq \gamma^{2}\xi^{\mathrm{T}}L_{n}(  \gamma_{0})  P_{n}%
L_{n}(  \gamma_{0})  \xi \\
&  \geq \gamma^{2}\lambda_{\min}(  L_{n}(  \gamma_{0})
P_{n}L_{n}(  \gamma_{0}) )  \Vert \xi  \Vert
^{2},\\
V_{e}(  t,e)   &  =\gamma^{2n+1}e^{\mathrm{T}}Q^{-1}(
\gamma)  e
\geq \gamma^{2}e^{\mathrm{T}}L_{n}\left(  \gamma_{0}\right)  Q_{n}^{^{-1}%
}L_{n}\left(  \gamma_{0}\right)  e\\
&  \geq \gamma^{2}\lambda_{\min}\left(  L_{n}\left(  \gamma_{0}\right)
Q_{n}^{^{-1}}L_{n}\left(  \gamma_{0}\right)  \right)  \left \Vert e\right \Vert
^{2},
\end{align*}
which indicates that%
\[
\mathcal{V}\left(  t,\xi,e\right)  =V_{\xi}\left(  t,\xi \right)  +k_{7}%
V_{e}\left(  t,e\right)  \geq \gamma^{2}\omega_{1}\left(  \gamma_{0}\right)
\left \Vert \chi \right \Vert ^{2},
\]
where $\chi=[\xi^{\mathrm{T}},e^{\mathrm{T}}]^{\mathrm{T}}$ and $\omega
_{1}(\gamma_{0})=\min \{ \lambda_{\min}(L_{n}(\gamma_{0})P_{n}$ $L_{n}%
(\gamma_{0})),\lambda_{\min}(L_{n}(\gamma_{0})Q_{n}^{^{-1}}L_{n}(\gamma
_{0}))\}$. Therefore, it yields that%
\begin{align}
\Vert \chi (  t) \Vert ^{2}  &  \leq \frac{\mathcal{V}%
(  t,\xi(t),e(t))  }{\gamma^{2}(  t)  \omega_{1}(
\gamma_{0})  }
\leq \frac{(  T-t)  ^{3}\mathrm{e}^{\lambda t}\mathcal{V}(
0,\xi(0),e(0))  }{\omega_{1}(  \gamma_{0})  }\nonumber \\
&  \leq \frac{(  T-t)  ^{3}\omega_{2}(  \gamma_{0})
}{T^{2}\omega_{1}(  \gamma_{0})  }\mathrm{e}^{\lambda t}\Vert
\chi(  0) \Vert ^{2}, \label{kexi}%
\end{align}
where $\omega_{2}(\gamma_{0})=\max \{ \lambda_{\max}(L_{n}(\gamma_{0}%
)P_{n}L_{n}(\gamma_{0})),\lambda_{\max}(L_{n}$ $(\gamma_{0})Q_{n}^{^{-1}}%
L_{n}(\gamma_{0}))\}$. Thus (\ref{F}) is proven. Moreover, according to
(\ref{z_i_x_i}) and $e=z-\xi$, it follows that%
\begin{align*}
\left \Vert x(t)  \right \Vert  &  \leq \left \Vert L_{n}(
T-t)  \right \Vert \left \Vert z(t)  \right \Vert \\
&  \leq \left \Vert L_{n}(  T-t)  \right \Vert \left(  \left \Vert
e(  0)  \right \Vert +\left \Vert \xi (0)  \right \Vert
\right) \\
&  \leq2\left \Vert L_{n}(T)  \right \Vert \left \Vert \chi \left(
t\right)  \right \Vert \\
&  \leq2\left \Vert L_{n}(  T)  \right \Vert \left(  T-t\right)
^{\frac{3}{2}}\frac{\sqrt{\omega_{2}\left(  \gamma_{0}\right)  }}%
{T\sqrt{\omega_{1}\left(  \gamma_{0}\right)  }}\mathrm{e}^{\frac{\lambda t}%
{2}}\left \Vert \chi \left(  0\right)  \right \Vert \\
&  \leq2\left \Vert L_{n}\left(  T\right)  \right \Vert \left(  T-t\right)
^{\frac{3}{2}}\frac{\sqrt{\omega_{2}\left(  \gamma_{0}\right)  }}%
{T\sqrt{\omega_{1}\left(  \gamma_{0}\right)  }}\\
&  \times \mathrm{e}^{\frac{\lambda t}{2}}\left(  \left \Vert L_{n}\left(
\frac{1}{T}\right)  \right \Vert \left \Vert x\left(  0\right)  \right \Vert
+2\left \Vert \xi \left(  0\right)  \right \Vert \right)  ,
\end{align*}
which proves (\ref{E}).

Finally, by using (\ref{bb}), (\ref{v_xi_e}) and (\ref{kexi}), we can obtain
that
\begin{align*}
\left \Vert u(  t)  \right \Vert ^{2}  &  =\beta^{2}\xi^{\mathrm{T}%
}P(  \gamma )  bb^{\mathrm{T}}P(  \gamma)  \xi \\
&  \leq \beta^{2}b^{\mathrm{T}}P(  \gamma )  b\xi^{\mathrm{T}%
}P(  \gamma )  \xi \\
&  =\beta^{2}n\gamma \xi^{\mathrm{T}}P(  \gamma)  \xi
  =\beta^{2}nV_{\xi}(  t,\xi) \\
&  \leq \beta^{2}n\mathcal{V}(  t,\xi,e) \\
&  \leq \beta^{2}n(  T-t)  \mathrm{e}^{\lambda t}\mathcal{V}(
0,\xi(0),e(0)) \\
&  \leq \beta^{2}n\gamma_{0}^{2}\omega_{2}(  \gamma_{0}) (
T-t)  \mathrm{e}^{\lambda t}\left \Vert \chi (  0)  \right \Vert
^{2}\\
&  \leq \beta^{2}n\gamma_{0}^{2}\omega_{2}(  \gamma_{0}) (
T-t)  \mathrm{e}^{\lambda t}\\
&  \times \left(  \left \Vert L_{n}\left(  \frac{1}{T}\right)  \right \Vert
\left \Vert x(  0)  \right \Vert +2\left \Vert \xi(  0)
\right \Vert \right)  ,
\end{align*}
which prove (\ref{G}). The proof is finished.
\end{proof}

For nonholonomic systems, compared to time-invariant discontinuous feedback,
such as \cite{Astolfi96scl,Ge03auto,Hongtac05,Han18tac} and their citations,
there are few related results on smooth time-varying feedback. To the best of
the authors' knowledge, there are no smooth time-varying feedback results for
prescribed-time control of nonholonomic systems available in the literature.

In this paper, on the one hand, different from the traditional discontinuous
control methods \cite{Astolfi96scl,Ge03auto,Hongtac05,Hou21ast,LQtac02}, the
proposed methods are smooth, time-varying and even linear. On the other hand,
different from smooth time-varying methods
\cite{Murry93tac,Teel95ijc,Tian00ieee,Tian02auto}, where only the state
feedback controllers were considered and drive the state to converge to zero
exponentially, the proposed controllers are not only based on state feedback
but also on output feedback, and can drive the state to converge to zero
globally at any prescribed time. In addition, different from
\cite{Tian00ieee,Tian02auto}, the parameters in our controllers do not depend
on the initial state of the $x_{0}$-subsystem.

\section{\label{sec4}A Numerical Example}

In this part, we consider the uncertain bilinear model
\cite{Ge03auto,Wang18auto}
\begin{equation}
\dot{x}_{0}=\left(  1-\frac{\varepsilon^{2}}{2}\right)  v,\quad \dot{z}_{1}%
=z_{2}v,\quad \dot{z}_{2}=u+\varphi \left(  z_{1}\right)  ,\label{sys_ex}%
\end{equation}
where $\varepsilon$ is the bias in orientation and is assumed to be much
smaller than one, $v$ and $u$ are the two controls, and $\varphi \left(
z_{1}\right)  =z\left(  1+\theta_{1}^{2}\right)  $ with $\theta_{1}=\theta
_{1}(t)$ being an uncertain function.

Denote the output $y=[x_{0},z_{1}]^{\mathrm{T}}$ and
\begin{align*}
u_{0}  &  =\left(  1-\frac{\varepsilon^{2}}{2}\right)  v,\quad x_{1}%
=z_{1},\quad x_{2}=\frac{2}{2-\varepsilon^{2}}z_{2},\\
u_{1}  &  =\frac{2}{2-\varepsilon^{2}}u,\quad \phi (  x_{1})
=\frac{2}{2-\varepsilon^{2}}\varphi (  x_{1})  ,
\end{align*}
the system (\ref{sys_ex}) can be expressed as
\[
\dot{x}_{0}=u_{0},\quad \dot{x}_{1}=x_{2}u_{0},\quad \dot{x}_{2}=u_{1}+\phi (
x_{1})  ,\quad y=[x_{0},x_{1}]^{\mathrm{T}}.
\]

For this system, the existing results are all based on the $\sigma$ process (see
\cite{Astolfi96scl}), that is, $x_{0}(t)$ appears on the denominator (see
\cite{Ge03auto,Wang18auto}). When $x_{0}(t)$ is close to zero, this may cause
singular problems. For the situation when $x_{0}(0)=0$, most of the existing
methods rely on using switching controls (see
\cite{Ge03auto,Hongtac05,Wang18auto}), and the control strategy proposed in this paper
is workable for all initial states of $x_{0}(t)$, thus having greater potential for application.


In the simulation, we take $\varepsilon=0.1$, $[x_{0}(0),z_{1}(0),z_{2}%
(0)]=[0,-1,1]$, $[\xi_{1}(0),\xi_{2}(0)]=[0,0]$, $\theta_{1}=\sin(t)$,
$T=2.5$s and $\beta=100$. The states and controllers are plotted in Fig.
\ref{figure1}, from which it can be observed that the proposed method can
indeed drive the states and controllers to converge to zero at the prescribed
time $T=2.5$s.

\begin{figure}[t]
\centering
\includegraphics[scale=0.55]{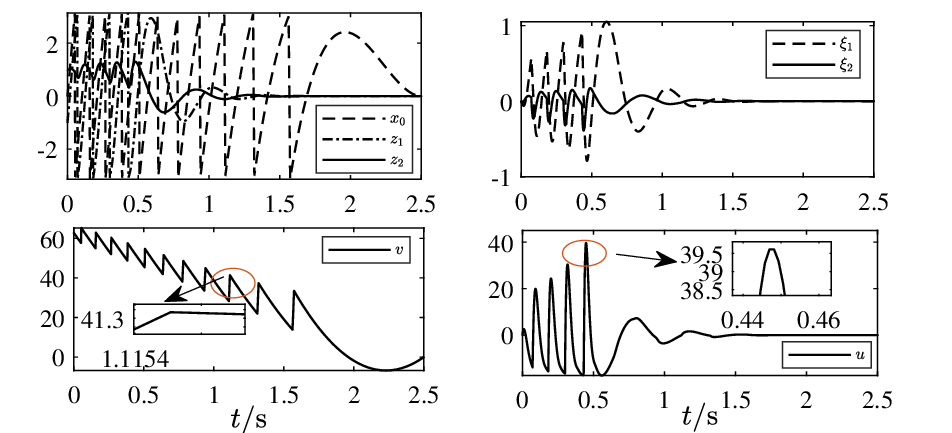}\caption{The state responses and control
signals for system (\ref{sys_ex})}%
\label{figure1}%
\end{figure}

\section{\label{sec5}Conclusion}

The smooth prescribed-time control problem for uncertain chained nonholonomic
systems was solved in this paper. With a novel smooth time-varying state
transformation, uncertain chained nonholonomic systems were reformulated as
uncertain linear time-varying systems. By fully utilizing the properties of a
class of parametric Lyapunov equations and constructing time-varying
Lyapunov-like functions, smooth time-varying high-gain controllers were
constructed. It was proved that the states and controllers can converge to zero at
the prescribed time. Both state feedback and observer-based output feedback
were considered. The effectiveness of the proposed methods was verified by a
numerical example.

\end{spacing}
\end{document}